\documentclass[12pt]{amsart}

\usepackage{enumerate, amsmath, amsthm, amsfonts, amssymb, xy}
\usepackage[usenames, dvipsnames]{color}
\usepackage[margin=1in]{geometry}
\usepackage[bookmarks, colorlinks=true, linkcolor=blue, citecolor=blue, urlcolor=blue]{hyperref}
\usepackage{tabu}
\xyoption{all}
\usepackage{stmaryrd}

\definecolor{prpl}{rgb}{0.7, 0.0, 0.7}

\newtheorem{theorem}{Theorem}

\newtheorem{proposition}[theorem]{Proposition}
\newtheorem{lemma}[theorem]{Lemma}
\newtheorem{corollary}[theorem]{Corollary}

\theoremstyle{definition}
\newtheorem{rmk}[theorem]{Remark}
\newenvironment{remark}[1][]{\begin{rmk}[#1] \pushQED{\qed}}{\popQED \end{rmk}}
\newtheorem{eg}[theorem]{Example}

\newtheorem{defn}[theorem]{Definition}

\newtheorem{ques}[theorem]{Question}
\newenvironment{question}[1][]{\begin{ques}[#1]\pushQED{\qed}}{\popQED \end{ques}}

\newcommand{\et}{\mathrm{et}}

\newcommand{\bA}{\mathbf{A}}
\newcommand{\cA}{\mathcal{A}}

\newcommand{\cB}{\mathcal{B}}

\newcommand{\cC}{\mathcal{C}}

\newcommand{\cE}{\mathcal{E}}

\newcommand{\bF}{\mathbf{F}}

\newcommand{\cI}{\mathcal{I}}

\newcommand{\cL}{\mathcal{L}}

\newcommand{\cM}{\mathcal{M}}

\newcommand{\cO}{\mathcal{O}}

\newcommand{\bQ}{\mathbf{Q}}

\newcommand{\rR}{\mathrm{R}}

\newcommand{\cT}{\mathcal{T}}

\newcommand{\cU}{\mathcal{U}}

\newcommand{\bZ}{\mathbf{Z}}

\DeclareMathOperator{\End}{End}

\DeclareMathOperator{\Spec}{Spec}

\newcommand{\GL}{\mathbf{GL}}

\DeclareMathOperator{\Hom}{Hom}

\DeclareMathOperator{\tr}{tr}
\DeclareMathOperator{\Res}{Res}

\let\wt\widetilde
\let\wh\widehat

\let\ol\overline

\newcommand{\lw}{{\textstyle \bigwedge}}

\title[{Constructing elliptic curves from Galois representations}]{Constructing elliptic curves from\\Galois representations}

\date{\today}

\author{Andrew Snowden}
\address{Department of Mathematics, University of Michigan, Ann Arbor, MI}
\email{\href{mailto:asnowden@umich.edu}{asnowden@umich.edu}}
\urladdr{\url{http://www-personal.umich.edu/~asnowden/}}
\thanks{AS was supported by NSF grants DMS-1303082 and DMS-1453893 and a Sloan Fellowship.}

\author{Jacob Tsimerman}
\address{Department of Mathematics, University of Toronto, Toronto, CA}
\email{\href{mailto:jacobt@math.toronto.edu}{jacobt@math.toronto.edu}}
\urladdr{\url{http://www.math.toronto.edu/~jacobt/}}

\begin{document}

\begin{abstract}
Given a non-isotrivial elliptic curve over an arithmetic surface, one obtains a lisse $\ell$-adic sheaf of rank two over the surface. This lisse sheaf has a number of straightforward
properties: cyclotomic determinant, finite ramification, rational traces of Frobenius, and somewhere not potentially good reduction. We prove that any lisse sheaf of rank two possessing these properties comes from an elliptic curve.
\end{abstract}
\maketitle

\section{Introduction}

Let $C/K$ be a proper, smooth geometrically irreducible curve, with $K$ a number field, let $f \colon E \to U\subset C$ be a non-isotrivial family of elliptic curves over a non-empty open subset $U$ of $C$, and let $L=\rR^1 f_*(\ol{\bQ}_{\ell})$ be the associated rank two lisse $\ell$-adic sheaf on $U$. The following properties hold:
\begin{enumerate}
\item There is an isomorphism $\lw^2 L \cong \ol{\bQ}_{\ell}(1)$.
\item There exists a proper smooth model $\cC$ of $C$ over $\Spec\cO_K[1/N]$, an open subset $\cU$ of $\cC$ extending $U$, and a lisse sheaf $\cL$ on $\cU$ extending $L$.
\item For every closed point $x$ of $\cU$, the trace of Frobenius on $\cL_x$ is a rational number.
\item There exists a point $x$ of $C_{\ol{K}}$ at which $L_{\ol{K}}$ does not have potentially good reduction, i.e., the inertia subgroup at $x$ does not act through a finite order quotient on the representation of $\pi_1^{\et}(U_{\ol{K}})$ associated to $L_{\ol{K}}$.
\end{enumerate}
To see (d), take $x$ to be a pole of the $j$-invariant of $E$.

The purpose of this paper is to prove the following converse of the above statement:

\begin{theorem}
\label{mainthm}
Let $C/K$ be as above, and let $L$ be an irreducible rank two lisse $\ol{\bQ}_{\ell}$-sheaf over an open subset $U\subset C$. Assume conditions (a)--(d) above hold. Then there exsits a family of elliptic curves $f \colon E \to U$ and an isomorphism $L \cong \rR^1 f_*(\ol{\bQ}_{\ell})$.
\end{theorem}

\begin{remark}
The Fontaine--Mazur conjecture predicts that representations of $G_K$ satisfying certain natural conditions should appear in the \'etale cohomology of algebraic varieties. It seems reasonable to expect some kind of generalization of this conjecture to higher dimensional bases. Our theorem can be viewed as confirmation of a very simple case of this.
\end{remark}

\begin{remark}
One can prove a version of Theroem~\ref{mainthm} where $C$ is replaced with a higher dimensional variety. We just treat the case of curves to keep the exposition simpler.
\end{remark}

\begin{remark}
The theorem is not true if one assumes only (a)--(c). Recall that a {\bf fake elliptic curve} is an abelian surface $A$ such that $\End(A)$ is an order $R$ in a non-split quaternion algebra that is split at infinity. The moduli space of fake elliptic curves corresponding to $R$ is a proper curve. We therefore can construct a family $f \colon A \to C$ of fake elliptic curves for some $C$ as above. The sheaf $\rR^1 f_* \ol{\bQ}_{\ell}$ decomposes as $L^{\oplus 2}$ for some rank two lisse sheaf $L$. This $L$ satisfies (a)--(c) but not (d), and thus does not come from a family of elliptic curves.
\end{remark}

\begin{question}
Suppose that for each prime number $\ell$ we have an irreducible rank two $\bQ_{\ell}$ sheaf $L_{\ell}$ satisfying (a)--(c) such that $\{L_{\ell}\}$ forms a compatible system (meaning that the $\cU$ in part (b) can be chosen uniformly and that the traces of Frobenius elements are independent of $\ell$). Does the system come from a family of elliptic curves? Note that the fake elliptic curve counterexample does not apply here: if $\ell$ ramifies in $R \otimes \bQ$ then $\rR^1 f_* \bQ_{\ell}$ does not decompose as $L^{\oplus 2}$.
\end{question}

\subsection{Summary of the proof}

The basic idea is to use Drinfeld's results on global Langlands to construct an elliptic curve over $\cC_{\bF_v}$ for most places $v$ of $\cO_K$, and then piece these together to get one over $\cC$. More precisely, we proceed as follows:
\begin{itemize}
\item We first show that we are free to pass to finite covers of $C$. The main content here is a descent result that shows that if $L$ comes from an elliptic curve over a cover of $C$ then it comes from an elliptic curve over $C$. Using this, we replace $C$ with a cover so that $L/\ell^3 L$ is trivial (after replacing $L$ with an integral form).
\item We next consider $L$ over $\cC_{\bF_v}$ and use Drinfeld's results on global Langlands to produce a $\GL_2$-type abelian variety $A_v$ realizing $L$.
\item Using hypotheses (c) and (d), we desdend the coefficient field of $A_v$ to $\bQ$, obtaining an elliptic curve $E_v$. (It is likely this could be obtained directly from Drinfeld's proof).
\item We next consider a certain moduli space $\cM$ of maps $\cU \to Y(\ell^3)$. From the previous step (and the triviality of $L/\ell^3 L$), we see that $\cM$ has $\bF_v$-points for infinitely many $v$. Since $\cM$ is of finite type over $\cO_K$, it therefore has a $\ol{K}$-point. This yields an elliptic curve $E_{\ol{K}}$ over $U_{\ol{K}}$ realizing $L_{\ol{K}}$.
\item Our hypotheses imply that $L_{\ol{K}}$ is irreducible. A simple representation theory argument thus shows that there is a finite extension $K'/K$ such that $E$ descends to $U_{K'}$ and its Tate module agrees with $L_{K'}$. We have already shown that it suffices to prove the result over a finite cover of $C$, so we are now finished.
\end{itemize}

\subsection{Outline}

In \S \ref{s:bg} we recall the relevant background material. In \S \ref{s:descent}, we prove a few descent results for abelian varieties. In \S \ref{s:drinfeld}, we package Drinfeld's results on global Langlands into the form we need; in particular, we use the results of \S \ref{s:descent} to produce elliptic curves (as opposed to $\GL_2$-type abelian varieties). In \S \ref{s:mapping}, we construct a mapping space parametrizing maps between two affine curves. Finally, in \S \ref{s:proof}, we prove Theorem~\ref{mainthm}.

\section{Background} \label{s:bg}

\subsection{Abelian varieties}

Let $A$ be an abelian variety over a field $K$ such that $\End_K(A) \otimes \bQ$ contains a number field $F$. Let $V_{\ell}(A)$ denote the rational Tate module of $A$ at the rational prime $\ell$. This is a module over $F \otimes \bQ_{\ell}=\prod_{w \mid \ell} F_w$, and thus decomposes as $\bigoplus_{w \mid \ell} V_w(A)$ where each $V_w(A)$ is a continuous representation of $G_K$ over the field $F_w$. We recall the following standard results:

\begin{proposition} \label{compsys}
Let $\sigma \in \End_K(A)$ commute with $F$. Then the characteristic polynomial of $\sigma$ on $V_w(A)$ (regarded as an $F_w$-vector space) has coefficients in $F$ and and is independent of $w$. In particular, each $V_w(A)$ has the same dimension over $F_w$.
\end{proposition}

\begin{proof}
See \cite[\S 11.10]{shimura} and (for $F=\bQ$) \cite[Proposition~9.23]{milne}.
\end{proof}

\begin{proposition} \label{Ffaltings}
Assume $K$ is a number field and $\End_K(A) \otimes \bQ=F$. Let $w$ be a place of $F$ above a prime $p$. Then $\End_{\bQ_p[G_K]}(V_w(A))=F_w$. In particular, $V_w(A)$ is absolutely irreducible as a representation of $G_K$ over $F_w$.
\end{proposition}

\begin{proof}
We have
\begin{displaymath}
F \otimes \bQ_p = \End_{\bQ_p[G_K]}(V_p(A))=\End_{\bQ_p[G_K]}(\bigoplus_{w \mid p} V_w(A)) \supset \bigoplus_{w \mid p} \End_{\bQ_p[G_K]}(V_w(A)) \supset \bigoplus_{w \mid p} F_w,
\end{displaymath}
where the first equality is the Tate conjecture proved by Faltings \cite[Theorem 1, page 211]{faltings}. Since the endmost spaces have the same dimension, we conclude that the containments are equalities, and so $\End_{\bQ_p[G_K]}(V_w(A))=F_w$.
\end{proof}

\subsection{Arithmetic fundamental groups}

Let $X$ be an affine normal integral scheme of finite type over $\bZ$ and consider $\pi_1^{\rm et}(X)$. For each closed point $x$ of $X$ there is a conjugacy class of Frobenius elements $F_x$. We recall the following generalization of the Chebotarev density theorem:

\begin{proposition} \label{chebo}
The elements $\{F_x\}_{x \in X}$ are dense in $\pi_1^{\rm et}(X)$.
\end{proposition}

\begin{proof}
This follows from \cite[Theorem~7]{serre}.
\end{proof}

\begin{corollary} \label{cheborep}
Suppose that $\rho_1$ and $\rho_2$ are semi-simple continuous representations $\pi_1^{\rm et}(X) \to \GL_n(\ol{\bQ}_{\ell})$ such that $\tr(\rho_1(F_x))=\tr(\rho_2(F_x))$ for all $x$. Then $\rho_1$ and $\rho_2$ are equivalent.
\end{corollary}

\subsection{Ramification of characters}

\begin{lemma} \label{charram}
Let $C$ be a curve over a number field $K$, and let $U$ be an open subset. Suppose that $\alpha \colon \pi_1(U) \to \ol{\bQ}_{\ell}^{\times}$ is a continuous homomorphism. Then for every point $x$ of $C_{\ol{K}}$, the inertia subgroup of $\pi_1(U_{\ol{K}})$ at $x$ has finite image under $\alpha$.
\end{lemma}

\begin{proof}
We are free to replace $K$ with a finite extension, so we may as well assume $x$ is a $K$-point. The decomposition group at $x$ has the form $\wh{\bZ} \rtimes G_K$, where $\wh{\bZ}$ is the geometric inertia group and $G_K$ acts on it through the cyclotomic character $\chi$. Let $T$ be a topological generator of $\wh{\bZ}$, written multiplicatively. Then for $\sigma \in G_K$ we have $\alpha(T)=\alpha(\sigma T \sigma^{-1})=\alpha(T^{\chi(\sigma)})=\alpha(T)^{\chi(\sigma)}$. It follows that $\alpha(T)$ has finite order.
\end{proof}

\subsection{Some lemmas from representation theory}

\begin{lemma}
\label{lemma:rep}
Let $\rho$ and $\rho'$ be finite dimensional representations of a group $G$ with equal determinant. Suppose there exists a normal subgroup $H$ of $G$ such that $\rho \vert_H$ and $\rho' \vert_H$ are absolutely irreducible and isomorphic. Then there exists a finite index subgroup $G'$ of $G$ such that $\rho \vert_{G'}$ and $\rho' \vert_{G'}$ are isomorphic.
\end{lemma}

\begin{proof}
Let $V$ and $V'$ be the spaces for $\rho$ and $\rho'$ and let $f \colon V \to V'$ be an isomorphism of $H$ representations. One easily verifies that for $g \in G$ the endomorphism $g^{-1} f^{-1} g f$ of $V$ commutes with $H$, and is therefore given by multiplication by some scalar $\chi(g)$. Thus we have $f^{-1} g f=\chi(g) g$ for all $g \in G$. It follows easily from this that $\chi$ is a homomorphism, i.e., $\chi(gg')=\chi(g) \chi(g')$. We thus see that $\chi \otimes \rho$ and $\rho'$ are isomorphic as representations of $G$. Taking determinants, we see that $\chi^n$ is trivial, where $n$ is the dimension of $\rho$. The result follows by taking $G'$ to be the kernel of $\chi$; this has finite index since the image of $\chi$ is contained in the group of $n$th roots of unity of $k$. (In fact, the index of $G'$ divides $n$.)
\end{proof}

\begin{lemma}\label{stupid1}
Let $G$ be a group and $H$ a normal subgroup of $G$. Consider a two-dimensional irreducible representation $\rho\colon G\to \GL_2(\ol{\bQ}_{\ell})$. Assume that for some element $h\in H$, the matrix $\rho(h)$ is a non-trivial unipotent element.
Then $\rho\vert_H$ is irreducible.
\end{lemma}

\begin{proof}
Suppose not. Then by the existence of $h$, there exists a unique one-dimensional
subspace $V$ invariant under $H$. But since $H$ is normal in $G$, it follows
that $V$ is invariant under $G$ as well. Contradiction.
\end{proof}

\section{Descent results for abelian varieties} \label{s:descent}

For this section, fix a finitely generated field $K$. We consider the following condition on a continuous representation $\rho \colon G_K \to \GL_n(\ol{\bQ}_{\ell})$:
\begin{itemize}
\item[($\ast$)] There exists an integral scheme $X$ of finite type over $\Spec(\bZ)$ with function field $K$ and a lisse $\ol{\bQ}_{\ell}$-sheaf $L$ on $X$ with generic fiber $\rho$ such that at every closed point $x$ of $X$ the trace of Frobenius on $L_x$ is rational.
\end{itemize}
In our proof of Theorem~\ref{mainthm}, we will show that $\rho$ comes from an elliptic curve over some finite extension of $K'$. We use the following result to conclude that we actually get an elliptic curve over $K$.

\begin{proposition}\label{extn}
Let $\rho \colon G_K \to \GL_2(\ol{\bQ}_{\ell})$ be a Galois representation satisfying condition ($\ast$). Suppose that there exists a finite extension $K'/K$ such that $\rho \vert_{K'}$ comes from a non-CM elliptic curve $E$. Then $\rho$ comes from an elliptic curve\footnote{The non-CM condition is in fact necessary. Let $K=\bF_p,K''=\bF_{p^2}$ and $\ell$ be such 
that $p$ has a square root mod $\ell$. Now let $\rho$ have eigenvalues $\pm\sqrt{p}$. Then $\rho$ 
can't come from an elliptic curve because the determinant is not the cyclotomic character but $\rho\vert_{K'}$ corresponds to a supersingular elliptic curve. }
.
\end{proposition}

We proceed with a number of lemmas.

\begin{lemma} \label{ratlrep}
Let $\rho \colon G \to \GL_n(\ol{\bQ}_{\ell})$ a continuous representation of the profinite group $G$. Suppose there exists an open subgroup $H$ of $G$ such that $\rho(H)$ is contained in $\GL_n(\bZ_{\ell})$ and contains an open subgroup of $\GL_n(\bZ_{\ell})$. Suppose also that there is a dense set of elements $\{g_i\}_{i \in I}$ of $G$ such that $\tr{\rho(g_i)} \in \bQ_{\ell}$ for all $i \in I$. Then $\rho(G)$ is contained in $\GL_n(\bQ_{\ell})$.
\end{lemma}

\begin{proof}
Let $e_{ij}$ be the $n \times n$ matrix with a~1 in the $(i,j)$ position and~0 elsewhere. By assumption, there exists $m$ such that $\rho(H)$ contains $1+\ell^m e_{ij}$ for all $i,j$, and so $e_{ij} \in \bQ_{\ell}[\rho(G)]$. For any $A \in \rho(G)$ we have $A_{i,j}=\tr(e_{ii}Ae_{jj}) \in \tr(\bQ_{\ell}[\rho(G)])=\bQ_{\ell}$, where the last equality follows from the fact that $\tr \circ \rho \colon G \to \bQ_{\ell}$ is continuous and $\tr(\rho(g_i)) \in \bQ_{\ell}$ for all $i$. It thus follows that $\rho(G)\subset \GL_n(\bQ_{\ell})$ as claimed.
\end{proof}

The following lemma is basically \cite[Corollary~2.4]{taylor}.

\begin{lemma} \label{gl2type}
Let $\rho \colon G_K \to \GL_2(\ol{\bQ}_{\ell})$ be a continuous irreducible Galois representation. Suppose that there is a finite separable extension $K'/K$ such that $\rho \vert_{G_{K'}}$ comes from a non-CM elliptic curve $E$. Then there is an abelian variety $A/K$ with $F=\End(A) \otimes \bQ$ is a number field of degree $\dim(A)$ such that $\rho \cong \ol{\bQ}_{\ell} \otimes_{F_w} V_w(A)$ for some place $w$ of $A$.
\end{lemma}

\begin{proof}
Consider $B=\Res_{K}^{K'}E$, an abelian variety over $K$ of dimension $[K':K]$ \cite[\S 7.6]{bosch}. Write $B=\prod_{i=1}^r B_i$ (up to isogeny) where each $B_i$ is a power of a simple abelian variety. Thus
\begin{displaymath}
\End_K(B)\otimes \bQ=\bigoplus_{i=1}^r D_i,
\end{displaymath}
where $D_i=\End_K(B_i) \otimes \bQ$ is a simple algebra. We have
\begin{displaymath}
\Hom_{G_K}(\rho, \ol{\bQ}_{\ell} \otimes V_{\ell}(B))=\Hom_{G_{K'}}(\rho \vert_{G_{K'}}, \ol{\bQ}_{\ell} \otimes V_{\ell}(E))=\ol{\bQ}_{\ell}
\end{displaymath}
where the latter follows from Falting's proof of the Tate conjecture \cite[Theorem 1, page 211]{faltings} since $E$ is non-CM. We thus see that $\rho$ occurs uniquely in $\ol{\bQ}_{\ell} \otimes V_{\ell}(B_i)$ for some $i$. Let $A$ be this $B_i$. Since $A$ is a power of a simple abelian variety and $\rho$ occurs uniquely in its Tate module, $A$ must itself be simple. The endomorphism ring $D_i$ must preserve $\rho$, and thus the action of $\rho$ comes from an algebra homomorphism $\psi \colon D_i \to \ol{\bQ}_{\ell}$, which is injective since $D_i$ is simple. We thus see that $D_i=F$ is a number field, and $\psi$ corresponds to a place $w$ of $F$ above $\ell$. Since $V_w(A) \otimes_{F_w} \ol{\bQ}_{\ell}$ contains $\rho$ and is absolutely irreducible by Proposition~\ref{Ffaltings}, it is equal to $\rho$. Therefore $V_w(A)$ is two dimensional over $F_w$, and so $[F:\bQ]=\dim(A)$ by Proposition~\ref{compsys}.
\end{proof}

\begin{lemma}
Proposition~\ref{extn} holds if $K'/K$ is separable.
\end{lemma}

\begin{proof}
By Lemma~\ref{gl2type}, we can find an abelian variety $A/K$ with $\End(A) \otimes \bQ=F$ a number field of degree $\dim(A)$, and a place $w_0$ of $F$ such that $\rho \cong \ol{\bQ}_{\ell} \otimes_{F_{w_0}} V_{w_0}(A)$.

Choose an integral scheme $X$ of finite type over $\Spec(\bZ)$ with fraction field $K$ and a family of abelian varieties $\cA \to X$ extending $A$. The representation of $G_K$ on $V_w(A)$ factors through $\pi_1^{\rm et}(X)$, for all $w$. Since $V_{w_0}(A)$ satisfies ($\ast$), we can replace $X$ with a dense open subscheme such that the Frobenius elements in $\pi_1^{\rm et}(X)$ have rational traces on $V_{w_0}(A)$. By Proposition~\ref{compsys}, it follows that the Frobenius elements in $\pi_1^{\rm et}(X)$ have rational traces on $V_w(A)$ for all $w$.

By assumption, $\rho \vert_{G_{K'}} \cong \ol{\bQ}_{\ell} \otimes V_{\ell}(E)$ for some non-CM elliptic curve $E/K'$. As above, pick an integral scheme $X'$ of finite type over $\Spec(\bZ)$ with
fraction field $K'$ such that there is an elliptic curve $\cE\to X'$ extending $E$. Further shrinking
$X,X'$ we can assume $X'$ maps to $X$ inducing the inclusion $K\subset K'$.
 
 Pick a place $w\mid p$ of $F$. As we saw above, Frobenius elements of $\pi_1^{\rm et}(X)$ have equal traces on $\rho \cong \ol{\bQ}_{\ell} \otimes_{F_{w_0}} V_{w_0}(A)$ and $\ol{\bQ}_p \otimes_{F_w} V_w(A)$. Similarly, the traces of Frobenius elements of $\pi_1^{\rm et}(X')$ on $V_{\ell}(E)$ and $V_p(E)$ are equal. We thus see by Corollary~\ref{cheborep} that $\ol{\bQ}_p \otimes_{F_w} V_w(A)$ is isomorphic to $\ol{\bQ}_p \otimes_{\bQ_p} V_p(E)$ as representations of $\pi_1^{\rm et}(X')$. By the Tate conjecture proved by Faltings \cite[Theorem 1, page 211]{faltings}, the image of $G_{K'}$ in $\GL(V_p(E))$ contains an open subgroup of $\GL_2(\bZ_p)$.

It follows that the conditions of Lemma~\ref{ratlrep} are fulfilled for 
$V_w(A)\otimes_{F_w}\ol{\bQ}_p$, and so $V_w(A)\otimes_{F_w}\ol{\bQ}_p$ is defined over $\bQ_p$; that is, there exists some representation $V$ of $G_K$ over $\bQ_p$ such that $V_w(A)\otimes_{F_w}\ol{\bQ}_p \cong \ol{\bQ}_p \otimes_{\bQ_p} V$. Since $V_w(A)$ and $V\otimes_{\bQ_p} F_w$ have the same
character, and are irreducible, it follows that they are isomorphic. We thus see that $\End_{\bQ_p[G_K]}(V_w(A)) \cong \End_{\bQ_p}(F_w)$, where on the right side we are taking endomorphisms of $F_w$ as a vector space. By Proposition~\ref{Ffaltings} we have that $\End_{\bQ_p[G_K]}(V_w(A)) \cong F_w$, and so  $\End_{\bQ_p}(F_w)=F_w$, which implies $F_w=\bQ_p$. We thus see that all places of $F$ are split, and so $F=\bQ$. Thus $A$ is actually an elliptic curve, and the proof is complete.
\end{proof}

\begin{lemma}
Proposition~\ref{extn} holds if $K'/K$ is purely inseparable.
\end{lemma}

\begin{proof}
It suffices to treat the case where $(K')^p = K$. Let $E/K'$ be the elliptic curve giving rise to $\rho \vert_{K'}$. Let $E^{(p)} = E \times_{K',F_0} K'$ where $F_0 \colon K' \to K'$ is the absolute Frobenius. Then $E^{(p)}$ is defined over $K$, and there is a canonical isogeny (relative Frobenius) $F \colon E \to E^{(p)}$ defined over $K'$ inducing an isomorphism on rational $\ell$-adic Tate modules. Thus $V_{\ell}(E^{(p)}) \vert_{G_{K'}} \cong \rho \vert_{G_{K'}}$ and so $V_{\ell}(E^{(p)}) \cong \rho$ since $K'/K$ is purely inseparable.
\end{proof}

Proposition~\ref{extn} in general follows from the previous two lemmas. We now prove a slightly different descent result.

\begin{proposition}\label{dcoef}
Let $k$ be a finite field, let $C/k$ be a proper smooth geometrically irreducible curve, and let $f \colon A \to U$ be a family of $g$-dimensional abelian varieties over a non-empty open subset $U$ of $C$ such that $\End(A) \otimes \bQ$ contains a number field $F$ of degree $g$. For a finite place $v$ of $F$, let $\cL_v$ be the $v$-adic Tate module of $A$. Assume that for all closed points $x$ of $U$, the trace of Frobenius at $x$ on $\cL_{v,x}$ belongs to $\bQ$. Also assume that there is some place of $C$ where $A$ does not have potentially good reduction. Then there exists an elliptic curve $E \to U$ such that $A$ is isogenous to $E \otimes \cO_F$.
\end{proposition}

\begin{proof}
Let $\ell$ be a rational prime that splits completely in $F$. Then the $\ell$-adic Tate module $\cL_{\ell}$ of $A$ decomposes as $\bigoplus_{v \mid \ell} \cL_v$. The $\cL_v$ form a compatible system with coefficients in $F$, so for each closed point $x \in U$ there exists $\alpha \in F$ such that the trace of Frobenius at $x$ on $\cL_{v,x}$ is the image of $\alpha$ in $F_v$. By our assumptions, $\alpha$ is a rational number, and so if $v,w \mid \ell$ then the traces of Frobenius on $\cL_{v,x}$ and $\cL_{w,x}$ are the same element of $\bQ_{\ell} \subset F_v, F_w$. It follows that the characters of $\cL_v$ and $\cL_w$ are equal at Frobenius elements, and so $\cL_v \cong \cL_w$. Thus $\dim(\End(\cL_{\ell})) \ge g^2$. By Faltings' isogeny theorem, it follows that $\dim(\End(A) \otimes \bQ) \ge g^2$.

Let $C' \to C$ be a cover such that $A$ has semi-stable reduction, and let $x$ be a point at which $A'$ (the pullback of $A$) has bad reduction. Let $\cA'$ be the N\'eron model of $A'$ over $C'$, and let $T$ be the torus quotient of the identity component of $\cA'_x$. The dimension $h$ of $T$ is at least~1, and at most~$g$. Under the map $\End(A) \otimes \bQ \to \End(T) \otimes \bQ \subset M_h(\bQ)$, the field $F$ must inject, and so $h=g$. Thus $T$ is the entire identity component of $\cA'_x$, and so the map $\End(A) \otimes \bQ \to \End(T) \otimes \bQ \subset M_g(\bQ)$ is injective. Combined with the preivous paragraph, we find that $\dim(\End(A) \otimes \bQ) = g^2$, and so the map $\End(A) \otimes \bQ \to M_g(\bQ)$ is an isomorphism. The statement now follows by projecting under an idempotent.
\end{proof}

\section{Drinfeld's work on global Langlands} \label{s:drinfeld}

\begin{proposition}
\label{drinfeld}

Let $k$ be a finite field, and $C/k$ be a smooth proper geometrically irreducicble curve. Let $L$ be an irreducible  rank 2 lisse $\ol{\bQ}_{\ell}$-sheaf over a non-empty open subset $U\subset C$, such that:
\begin{enumerate}
\item There is an isomorphism $\lw^2 L \cong \ol{\bQ}_{\ell}(1)$.
\item For every closed point $x$ of $C$, the trace of Frobenius on $L_x$ is a rational number.
\item There exists a point $x$ of $C_{\ol{k}}$ at which $L_{\ol{k}}$ does not have potentially good reduction.
\end{enumerate}

Then there exists an elliptic curve $f\colon E\to U$ such that 
$\rR^1f_{*}(\ol{\bQ}_{\ell})\cong L$.

\end{proposition}

\begin{proof}

By Drinfeld's theorem (\cite[Main Thm, remark 5]{D1}, see also \cite{D3}) there
is a cuspidal automorphic representation $\pi$ of $\GL_2(\bA_{k(C)})$ which is 
compatible with $L$. Since inertia at $x$ does not have finite order, $\pi$
must be special at $x$. It follows by another theorem of Drinfeld \cite[Thm.~1]{D2}
that there exists a number field $E$ and a $\GL_2(E)$-type abelian variety $A$
over $U$ which is compatible with $\pi$ and thus also with $L$, in the sense that
the $\ell$-adic Tate module $L'$ of $A$ is isomorphic with $L$ when tensored up
to $\ol{\bQ}_{\ell}$; that is, $L'\otimes_E \ol{\bQ}_{\ell}\cong L$. 
By Proposition~\ref{dcoef}, we may may take $A$ to be an elliptic curve.

\end{proof}

\begin{remark}

By following Drinfeld's proof carefully, one may directly see that we can take
$E$ to be the field generated by the Frobenius traces of $L$, and is thus $\bQ$,
which can replace the use of Proposition~\ref{dcoef}.

\end{remark}

\section{Mapping spaces} \label{s:mapping}

Let $S$ be a noetherian scheme. For $i=1,2$, let $C_i$ be a proper smooth scheme over $S$ with geometric fibers irreducible curves of genus $\ge 2$, let $Z_i$ be a closed subscheme of $C_i$ that is a finite union of sections $S \to C_i$, and let $U_i$ be the complement of $Z_i$ in $C_i$. Fix $d \ge 1$.

\begin{proposition}
\label{prop:mapping}
There exists a scheme $\cM$ of finite type over $S$ and a map $\phi \colon (U_1)_{\cM} \to (U_2)_{\cM}$ with the following property: if $k$ is a field, $s \in S(k)$, and $f \colon U_{1,s} \to U_{2,s}$ is a map of curves over $k$ of degree $d$ (meaning the corresponding function field extension has degree $d$), then there exists $t \in \cM(k)$ over $s$ such that $f=\phi_t$.
\end{proposition}

\begin{proof}
Let $P$ be the image of a section $S \to C_1$. Define $\wt{P}$ to be the $d$th nilpotent thickening of $P$. Precisley, if $P$ is defined by the ideal sheaf $\cI_P$ then $\wt{P}$ is defined by $\cI^d_P$. Let $Q$ be the image of a section $S \to C_2$, and define $\wt{Q}$ similarly. Let $0 \le e \le d$ be an integer. For an $S$-scheme $S'$, let $\cC_{P,Q,e}(S')$ be the set of maps $f \colon \wt{P}_{S'} \to \wt{Q}_{S'}$ such that $f^*(\cI_Q) \subset \cI_P^e$. As these are finite schemes over $S$, this functor is represented by a scheme $\cC_{P,Q,E}$ of finite type over $S$.

Write $Z_1=\coprod_{i=1}^n P_i$ and $Z_2=\coprod_{j=1}^m Q_j$. By a \emph{ramification datum} we mean a tuple $\rho=(\rho_1, \ldots, \rho_n)$ where each $\rho_i$ is either null (denoted $\emptyset$), or a pair $(k_i, e_i)$, where $1 \le k_i \le m$ and $0 \le e_i \le d$, such that the following condition holds: for any $1 \le j \le m$, we have $\sum_{k_i=j} e_i = d$ (the sum taken over $i$ for which $\rho_i \ne \emptyset$). Obviously, there are only finitely many ramification data. For a ramification datum $\rho$, define $\cC_{\rho} = \prod_{\rho_i \ne \emptyset} \cC_{P_i, Q_{k_i}, e_i}$. Finally, define $\cC=\coprod_{\rho} \cC_{\rho}$, where the union is taken over all ramification data $\rho$.

For an $S$-scheme $S'$, let $\cA(S')$ be the set of morphisms $(C_1)_{S'} \to (C_2)_{S'}$ having degree $d$ in each geometric fiber. The theory of the Hilbert scheme shows that $\cA$ is represented by a scheme of finite type over $S$. For $1 \le i \le n$, let $\cB_i(S')$ be the set of all maps $(\wt{P}_i)_{S'} \to (C_2)_{S'}$. This is easily seen to be a scheme of finite type over $S$. Let $\cB=\coprod_{i=1}^n \wt{\cB}_i$.

We have restriction maps $\cA \to \cB$ and $\cC \to \cB$. Define $\cM$ to be the fiber product $\cA \times_{\cB} \cC$, which is a scheme of finite type over $S$. We can write $\cM=\coprod_{\rho} \cM_{\rho}$, where $\cM_{\rho}=\cA \times_{\cB} \cC_{\rho}$. If $s \in S(k)$ then $\cM_{\rho,s}(k)$ is the set of degree $d$ maps $f \colon C_{1,s} \to C_{2,s}$ satisfying the following conditions at the $P_i$: if $\rho_i=\emptyset$ then there is no condition at $P_i$; otherwise, $f(P_i)=Q_{k_i}$ and the ramification index $e(P_i \mid Q_{k_i})$ is at least $e_i$. (In fact, the ramification index is exactly $e_i$, since the total ramification is $d$ and the $e_i$ with $k_i=j$ add up to $d$.) It is clear that such a map carries $U_{1,s}$ into $U_{2,s}$, and that any map $U_{1,s} \to U_{2,s}$ of degree $d$ comes from a point of some $\cM_{\rho,s}$. Thus every map $f \colon U_{1,s} \to U_{2,s}$ comes from some $k$-point of $\cM_s$. Finally, note that the universal map $(C_1)_{\cM} \to (C_2)_{\cM}$ carries $(U_1)_{\cM}$ to $(U_2)_{\cM}$, as this can be checked at field points of $\cM$. This proves the proposition.
\end{proof}

\section{Proof of Theorem~\ref{mainthm}} \label{s:proof}

Keep notation as in Theorem~\ref{mainthm}, and put $S=\Spec(\cO_K[1/N])$.

\begin{lemma} \label{irredsubgp}
The restriction of $L$ to any open subgroup of $\pi_1^{\et}(U_K)$ is irreducible.
\end{lemma}

\begin{proof}
Suppose not. Then there exists an open normal subgroup $H$ of $G=\pi_1^{\et}(U_K)$ such that $L \vert_H$ is reducible. It is semi-simple by Lemma~\ref{stupid1}, and therefore a sum of two characters. Since characters have finite ramficiation by Lemma~\ref{charram}, we have contradicted assumption (d).
\end{proof}

Choose $r \gg 0$ so that $X=X(\ell^r)$ has genus at least~2 and $Y=Y(\ell^r)$ is a fine moduli space; in fact, $r=3$ suffices for any $\ell$. We replace $\cL$ with a rank two $\cO_E$-sheaf, where $E/\bQ_{\ell}$ is a finite extension. Proposition~\ref{extn} and Lemma~\ref{irredsubgp} show that is suffices to prove the proposition after passing to a finite cover of $C$. By passing to an appropriate cover, we can therefore assume that $\cL/\ell^r$ is trivial. The image of the Galois representation $\rho \colon \pi_1^{\et}(\cU) \to \GL_2(\cO_E)$ has order $M \cdot \ell^{\infty}$ (in the sense of profinite groups) for some positive integer $M$. By enlarging $N$, we can assume that $M \cdot \ell \mid N$ and that the complement of $U$ in $C$ spreads out to a divisor on $\cC$ that is smooth over $S$

We make the following definitions:
\begin{itemize}
\item  Let $D$ be an integer greater than $(g(C)-1)/g(X)-1)$, where $g(-)$ denotes genus.
Let $\cM_d$ be the space of maps $\cU \to Y$ of degree $d$, in the sense of Proposition~\ref{prop:mapping}, and let $\cM=\coprod_{d=1}^D \cM_d$.
\item Let $\cT$ be the (integral) $\ell$-adic Tate module of the universal elliptic curve over $Y$, and let $\cT'=\cT \otimes \cO_E$. Also let $\cL_n=\cL/\ell^n \cL$ and let $\cT'_n=\cT'/\ell^n \cT'$.
\item Let $\wt{\cM}_n$ be the moduli space of pairs $(f, \psi)$ where $f \in \cM$ and $\psi$ is an isomorphism of $\cO_E$-sheaves $f^*(\cT_n) \to \cL_n$. 
\end{itemize}

\begin{lemma} \label{lem:proper}
The map $\pi:\wt{\cM}_n \to \cM$ is finite.
\end{lemma}

\begin{proof}
Note that for every field-valued point $f$ of $\cM$ there are only finitely many choices for $\psi$ and so the map $\wt{\cM}_n\to\cM$ is quasi-finite. Thus, to prove the lemma it is sufficent to show that $\pi$ is proper.

We use the valuative criterion. Let $R$ be a DVR with fraction field $F$. Let $f\in\cM(R)$ and 
$(\psi,f)\in\wt{\cM}_n(F)$. Thus, $f$ corresponds to a map $f:U_R\to Y_R$ and 
$\psi: f^*(\cT_n)_F\to(\cL_n)_F$ is an isomorphism. Since $f^*(\cT_n)$ and $\cL_n$ are finite etale sheaves on $U_R$, which is a normal scheme, $\psi$ extends uniquely over $U_R$. 
\end{proof}

Let $\cM_n$ be the image of $\wt{\cM}_n$ in $\cM$, which is closed by Lemma~\ref{lem:proper}. We endow it with the reduced subscheme structure. As the $\cM_n$ form a descending chain of closed subschemes of $\cM$, they stabilize. Let $\cM_{\infty}$ be $\cM_n$ for $n \gg 0$.

\begin{lemma}
The fiber of $\cM_{\infty}$ over all closed points of S is non-empty.
\end{lemma}

\begin{proof}
Let $s$ be a closed point of $S$ of characteristic $p$. By Lemma~\ref{stupid1}, $\cL_{\ol{K}}$ is an irreducible sheaf. Let
$\ol{S}$ be the strict Hensilization of $S$ at $s$,  $\ol{s}$ the geometric point corresponding to $s$, and $K_s$ the fraction field of $\ol{S}$. By \cite[XIII,2.10,pg.~289]{sga},
\begin{displaymath}
\pi_1^{\et,(p)}(\cU_{\ol{s}})\cong \pi_1^{\et,(p)}(\cU_{\ol{S}})\twoheadleftarrow \pi_1^{\et}(U_{K_s})\supset \pi_1^{\et}(U_{\ol{K_s}})=\pi_1^{\et}(U_{\ol{K}}) \leqno{(*)}
\end{displaymath}
where $\pi_1^{\et, (p)}$ denotes the prime to $p$ quotient of $\pi_1^{\et}$. We can regard $\cL$ as a representation of $\pi_1^{\et}(\cU)$ and then restrict it to a representation of $\pi_1^{\et}(\cU_{\ol{S}})$; since the image of the representation has order $M \ell^{\infty}$, which is prime to $p$, it factors through $\pi_1^{\et,(p)}(\cU_{\ol{S}})$. We thus obtain a representation of $\pi_1^{\et,(p)}(\cU_{\ol{s}})$. The pullback of this representation to $\pi_1^{\et}(U_{\ol{K}})$ is irreducible. It follows now that $\cL_s$ is an irreducible sheaf on $\cU_s$.

By Proposition~\ref{drinfeld}, we can find a family of elliptic curves $\cE \to \cU_s$ and an isomorphism $\cL_s \vert_{\cU_s} \cong T(\cE) \otimes \cO_E$, where $T(\cE)$ is the relative Tate module of $\cE$. It follows that $T(\cE)/\ell^r$ is the trivial sheaf, and so we can find a basis for $\cE[\ell^r]$ over $\cU_s$. We thus have a map $f \colon \cU_s \to Y$ such that $T(\cE) \cong f^*(\cT)$. Factor $f$ as $g \circ F^n$, where $g \colon \cU_s \to Y$ is separable and $F \colon Y_{\kappa(s)} \to Y_{\kappa(s)}$ is the absolute Frobenius. Note that $F^*(\cT) \cong \cT$, and so $T(\cE) \cong g^*(\cT)$. Let $\ol{g}$ be the extension of $g$ to a map $C_s \to X$. Since $\ol{g}$ is separable, it has degree $\le D$. Thus $\ol{g}$, and the isomorphism $\cL \cong g^*(\cT) \otimes \cO_E$, define a $\kappa(s)$ points of $\cM_n$ for all $n$, which proves the lemma.
\end{proof}

\begin{proof}[Proof of Theorem~\ref{mainthm}]
Since $\cM_{\infty}$ is finite type over $S$ and all of its fibers over closed points are non-empty, it follows that the generic fiber of $\cM_{\infty}$ is non-empty. Choose a point $x$ in $\cM_{\infty}(L')$, for some finite extension $L'/L$, corresponding to a family of elliptic curves $\cE \to U_{L'}$.  Now, $x$ lifts to $\wt{\cM}_n(\ol{L})$ for all $n$. Thus $\cL_n$ and $T(\cE) \otimes \cO_E/\ell^n$ are isomorphic for all $n$, as sheaves on $U_{\ol{L}}$. It follows that $\cL$ and $T(\cE) \otimes \cO_E$ are isomorphic over $U_{\ol{L}}$ (by compactness). By Lemma~\ref{lemma:rep}, $\cL[1/\ell]$ and $T(\cE) \otimes E$ are isomorphic over $U_{L''}$, for some finite (even quadratic) extension $L''$ of $L'$. Passing to the generic fibers, we see that $\rho \vert_{KL''}$ comes from an elliptic curve, which completes the proof by Proposition~\ref{extn}.
\end{proof}

\end{document}